\documentclass[12pt]{amsart}
\usepackage{amsfonts,amssymb,amscd,amstext,mathrsfs,color}

\input xy
\xyoption{all}
 
\newtheorem{theorem}{Theorem}
\newtheorem{corollary}{Corollary}
\newtheorem{lemma}{Lemma}
\newtheorem{proposition}{Proposition}

\theoremstyle{definition}
\newtheorem{remark}{Remark}

\newcommand{\C}{\mathbb{C}}
\newcommand{\End}{{\operatorname{End}}}
\newcommand{\hyp}{{\operatorname{hyp}}}
\newcommand{\att}{{\operatorname{att}}}
\newcommand{\sad}{{\operatorname{sad}}}
\newcommand{\rep}{{\operatorname{rep}}}
\newcommand{\clo}{{\operatorname{clo}}}
\newcommand{\rne}{{\operatorname{rne}}}

\begin{document}

\title{Dynamics of generic endomorphisms \\ of Oka-Stein manifolds}

\author{Leandro Arosio and Finnur L\'arusson}

\address{Dipartimento Di Matematica, Universit\`a di Roma \lq\lq Tor Vergata\rq\rq, Via Della Ricerca Scientifica 1, 00133 Roma, Italy}
\email{arosio@mat.uniroma2.it}

\address{School of Mathematical Sciences, University of Adelaide, Adelaide SA 5005, Australia}
\email{finnur.larusson@adelaide.edu.au}



\thanks{L.~Arosio was supported by SIR grant \lq\lq NEWHOLITE -- New methods in holomorphic iteration\rq\rq, no.~RBSI14CFME, and partially supported by the MIUR Excellence Department Project awarded to the Department of Mathematics, University of Rome Tor Vergata, CUP E83C18000100006.  F.~L\'arusson was partially supported by Australian Research Council grant DP150103442.  Part of this work was done when F.~L\'arusson visited Rome in February 2020.  He thanks the University of Rome Tor Vergata for financial support and hospitality.}

\date{1 February 2021.  Minor edits 4 March 2021}


\begin{abstract}
We study the dynamics of a generic endomorphism $f$ of an Oka-Stein manifold $X$.  Such manifolds include all connected linear algebraic groups and, more generally, all Stein homogeneous spaces of complex Lie groups.  We give several descriptions of the Fatou set and the Julia set of $f$.  In particular, we show that the Julia set is the derived set of the set of attracting periodic points of $f$ and that it is also the closure of the set of repelling periodic points of $f$.  Among other results, we prove that $f$ is chaotic on the Julia set and that every periodic point of $f$ is hyperbolic.  We also give an explicit description of the \lq\lq Conley decomposition\rq\rq\ of $X$ induced by $f$ into chain-recurrence classes and basins of attractors.  For $X=\C$, we prove that every Fatou component is a disc and that every point in the Fatou set is attracted to an attracting cycle or lies in a dynamically bounded wandering domain (whether such domains exist is an open question).
\end{abstract}

\maketitle

\section{Introduction and main theorem} 
\label{sec:intro}

\noindent
This paper continues a line of research begun with our previous papers \cite{AL2019} and \cite{AL2020}.  Here, we investigate the dynamics of a generic endomorphism of an Oka-Stein manifold.\footnote{By an Oka-Stein manifold we simply mean a complex manifold that is both Oka and Stein.}  The setting is very general and includes all Stein homogeneous spaces of complex Lie groups, in particular all connected linear algebraic groups.  Even in the much-studied case of complex affine space $\C^n$, $n\geq 1$, most of our results are new.  The classes of Oka manifolds and Stein manifolds are fundamental in modern complex analysis and geometry.  Manifolds in their intersection have the right degree of flexibility to possess rich dynamics.

The family of endomorphisms of an Oka-Stein manifold is so large and diverse that little can be said about its dynamics without restricting the analysis to suitable subfamilies that are usually taken to be quite small (for example the family of regular polynomial endomorphisms of $\C^n$, as in \cite{BJ2000}). One basic property that fails in general is the complete invariance of the Fatou and Julia sets.  We show that this property, and many other interesting dynamical properties, are generic with respect to the compact-open topology, which is the only natural topology in this context.  Hence, somewhat surprisingly, strong dynamical theorems hold for very large subfamilies of endomorphisms of Oka-Stein manifolds.

We give several descriptions of the Fatou set and the Julia set of a generic endomorphism $f$ of an Oka-Stein manifold $X$.  In particular, we show that the Julia set is the derived set of the set of attracting periodic points of $f$ and that it is also the closure of the set of repelling periodic points of $f$.  Among other results, we prove that $f$ is chaotic on the Julia set and that every periodic point of $f$ is hyperbolic (this is one half of the Kupka-Smale theorem in our setting).  We also give an explicit description of the \lq\lq Conley decomposition\rq\rq\ of $X$ induced by $f$ into chain-recurrence classes and basins of attractors.  For $X=\C$, we prove that every Fatou component is a disc and that every point in the Fatou set is attracted to an attracting cycle or lies in a dynamically bounded wandering domain (whether such domains exist is an open question).

Before stating our first result, we need to clarify some definitions and establish some notation.  We denote by $\End\,X$ the monoid of endomorphisms of a complex manifold $X$, that is, holomorphic maps $X\to X$, with the compact-open topology, which is separable and defined by a complete metric, so $\End\,X$ is a Polish space.
\begin{itemize}
\item  $\hyp(f)$ is the set of hyperbolic periodic points of an endomorphism $f$ of $X$.  A periodic point $p$ of $f$ of period $n$ is hyperbolic if the derivative $D_p f^n$ of the $n^\textrm{th}$ iterate $f^n$ at $p$ has no eigenvalue of absolute value $1$.
\item  $\att(f)$ is the set of attracting periodic points of $f$.  The periodic point $p$ is attracting if all the eigenvalues of $D_p f^n$ have absolute value less than $1$.
\item  The periodic point $p$ is super-attracting if $D_p f^n=0$ (we consider a super-attracting point to be attracting).
\item  $\rep(f)$ is the set of repelling periodic points of $f$.  The periodic point $p$ is repelling if all the eigenvalues of $D_p f^n$ have absolute value greater than $1$.
\item  $\sad(f)$ is the set of saddle periodic points of $f$.  The periodic point $p$ is a saddle point if it is hyperbolic and some of the eigenvalues of $D_p f^n$ have absolute value less than $1$ and some have absolute value greater than $1$.
\item  $p\in X$ is an escaping point of $f$ if $f^n(p)$ converges to the point at infinity as $n\to\infty$.
\item  $\clo(f)$ is the set of points $p\in X$ whose $f$-orbit can be closed in the weak sense, meaning that for every neighbourhood $U$ of $p$ in $X$ and $V$ of $f$ in $\End\, X$, there is an endomorphism in $V$ with a periodic point in $U$.  Note that $\clo(f)$ is a closed subset of $X$.  The $f$-orbit of $p$ can be closed in the strong sense if for every neighbourhood $V$ of $f$ in $\End\, X$, there is an endomorphism in $V$ for which $p$ itself is periodic.  If $X$ is homogeneous in the weak sense that the automorphism group of $X$ acts transitively on $X$, as is the case for all known Oka-Stein manifolds, then the two notions coincide.
\item  $\Omega_f$ is the non-wandering set of $f$, that is, the set of points $p\in X$ such that for every neighbourhood $U$ of $p$, there is $k\geq 1$ such that $U\cap f^k(U)\neq\varnothing$.  Note that $\Omega_f$ is a closed subset of $X$.
\item  $F_f$ is the Fatou set of $f$, the open set of normality of the iterates of $f$.  More explicitly, $F_f$ is the set of points in $X$ with a neighbourhood $U$ such that every subsequence of the sequence of iterates of $f$ has a subsequence that converges locally uniformly on $U$ to a holomorphic map into $X$ or to the point at infinity.  Clearly, $F_f$ is backward invariant, that is, $f^{-1}(F_f)\subset F_f$.
\item  $J_f$ is the Julia set of $f$, defined as the complement of $F_f$ in $X$.  Since $F_f$ is backward invariant, $J_f$ is forward invariant, that is, $f(J_f)\subset J_f$.
\item  $\rne(f)$ is the open set of points $p\in X$ at which $f$ is robustly non-expelling, meaning that there are neighbourhoods $U$ of $p$ in $X$ and $V$ of $f$ in $\End\, X$ and a compact subset $K$ of $X$ such that $g^j(U)\subset K$ for all $g\in V$ and $j\geq 0$.  If $X$ is Stein, by Montel's theorem, $\rne(f)\subset F_f$.
\end{itemize}
The following results from our previous paper \cite{AL2020}, building on the groundbreaking work of Forn\ae ss and Sibony \cite{FS1997}, will be important here.  If $f$ is an endomorphism of an Oka-Stein manifold $X$, then 
\[ \rne(f)\cap\Omega_f = \att(f) \]
and
\[ \Omega_f \subset \clo(f) \qquad \textrm{(closing lemma)}. \]
Also, for a generic endomorphism $f$ of $X$, that is, an endomorphism in a suitable residual subset of $\End\, X$, 
\[ \overline{\hyp(f)} = \Omega_f \qquad \textrm{(general density theorem)}. \]

Our main theorem is as follows.  It is proved in Section \ref{sec:proof}.

\begin{theorem}   \label{t:main-theorem}
A generic endomorphism $f$ of an Oka-Stein manifold $X$ has the following properties.
\begin{enumerate}
\item[(a)]  Every periodic point of $f$ is hyperbolic.
\item[(b)]  The Fatou set is completely invariant, that is, $f^{-1}(F_f)=F_f$, and its connected components are Runge (in particular Stein).  Moreover,
\[F_f=\rne(f).\]
When $X$ is homogeneous, $F_f$ consists of those points of $X$ whose orbit under every sufficiently small perturbation of $f$ is relatively compact. 
\item[(c)]   The Julia set $J_f$ is completely invariant, perfect, not compact, and has no interior, so $J_f$ and $F_f$ are not empty.  Also, $J_f$ is the boundary of the set of points in $X$ with relatively compact $f$-orbit.  Moreover,
\[ J_f = \overline{\rep(f)} = \overline{\att(f)}\setminus \att(f), \]
and when $\dim X\geq 2$, 
\[ J_f = \overline{\sad(f)}. \]
When $X$ is homogeneous, $J_f$ consists of those points of $X$ that are escaping with respect to arbitrarily small perturbations of $f$.
\item[(d)]  \[\clo(f) = \Omega_f = \overline{\att(f)}=J_f\cup\att(f).\]
\item[(e)]  $f$ is chaotic on $J_f$.  If $U$ is a neighbourhood of a point in $J_f$, then $\bigcup\limits_{n\geq 0} f^n(U)$ is dense in $X$.  Every basin of attraction of an attracting cycle has $J_f$ as its boundary.
\item[(f)]  The iterates of $f$ have the same Fatou set and Julia set as $f$ itself.
\end{enumerate}
\end{theorem}

\begin{remark}
(a)  Perhaps the most novel part of the theorem is the inclusion $J_f\subset\overline{\att(f)}$ for a generic endomorphism $f$.  As far as we know, this result is new even for $X=\C^n$.  It implies, for example, that any neighbourhood of a point in $J_f$ intersects infinitely many basins of attraction.  It follows that $f$ has infinitely many Fatou components.  We can think of $\att(f)$ as an infinite discrete metric space on which $f$ acts in such a way that every point has a finite orbit.  This dynamical system determines $J_f$ and the action of $f$ on~it.

(b)  If $\dim X\geq 2$ and $1\leq k\leq \dim X-1$, then $J_f$ is the closure of the set of saddle periodic points of $f$ at which the derivative of the appropriate iterate of $f$ has $k$ eigenvalues $\lambda$ with $\lvert\lambda\rvert<1$ and $\dim X-k$ eigenvalues $\lambda$ with $\lvert\lambda\rvert>1$.  This is evident from the proof of Theorem \ref{t:key-closing-lemma}.

(c)  The most important classes of Oka-Stein manifolds are connected linear algebraic groups and, more generally, Stein homogeneous spaces of complex Lie groups.  Stein manifolds with the density property or the volume density property are Oka.  The very special examples $\C^m\times\C^{*n}$, $m,n\geq 0$, $m+n\geq 1$, are of course of great interest.  There are not many known ways to produce new Oka-Stein manifolds from old.  The two properties do not go well together in this respect.  All we can say is that products, retracts, and covering spaces of Oka-Stein manifolds are Oka-Stein.  (The definitive reference on Oka theory is \cite{Forstneric2017}.  See also the survey \cite{FL2011}.)

(d)  In the proof of part (a) of the theorem, we use a dominating spray on $X$, given by the fact that for a Stein manifold, the Oka property and Gromov's ellipticity are equivalent.  In the proofs of parts (b)--(f), we directly apply the Oka property, more specifically the basic Oka property with approximation and jet interpolation.  It is one of the nontrivially equivalent formulations of the Oka property and says the following for a complex manifold $X$.

Let $S$ be a reduced Stein space, $K$ be a holomorphically convex compact subset of $S$, and $T$ be a closed analytic subvariety of $S$.  Every continuous map $S\to X$ that is holomorphic on a neighbourhood of $K\cup T$ can be deformed to a holomorphic map by a deformation that is arbitrarily small on $K$ and fixed to any finite order along $T$.

When $X$ is Oka and Stein, we can take $S=X$ and it is immediately clear that the monoid of endomorphisms of $X$ is very large.  For example, we can prescribe the values of an endomorphism of $X$ on any discrete subset of $X$.

(e)  A residual set intersected with an open set is residual in the open set, so results about generic endomorphisms of an Oka-Stein manifold $X$ also apply to open sets of endomorphisms.  When $X$ is affine algebraic, as the most important examples are, and therefore strongly deformation-retracts onto a compact subset of itself \cite[Theorem 1.1]{HM1997}, every homotopy class of endomorphisms is open.  So are unions of homotopy classes, of course, such as the set of all endomorphisms that induce a given endomorphism of a particular homotopy, homology, or cohomology group.

(f)  {\it Non-hyperbolicity is generic.}  While there is no standard definition of hyperbolic dynamics on a noncompact set, having both repelling periodic points and saddle periodic points dense in the Julia set rules out a generic endomorphism of an Oka-Stein manifold of dimension at least 2 being hyperbolic on its Julia set in any reasonable sense.

In the 1-dimensional case, Rempe-Gillen and Sixsmith have proposed a definition of hyperbolicity for transcendental entire functions \cite[Theorem and Definition 1.3]{RS2017}.  Their hyperbolicity property implies that the set of critical values is bounded.  For both $\C$ and $\C^*$, using the persistence of isolated critical points, an argument similar to the proof of Proposition \ref{p:more-about-Julia}(a) shows that the set of critical values of a generic endomorphism is not relatively compact.
\end{remark}

The remainder of the paper is organised as follows. In the next section we prove Theorem \ref{t:main-theorem}.  In Section \ref{sec:Fatou}, we
present two results about Fatou components of generic endomorphisms, the first for the complex plane $\C$ and the second for an arbitrary Oka-Stein manifold.  In the final Section \ref{sec:Conley}, we consider Conley's fundamental theorems on general dynamical systems in the setting of an Oka-Stein manifold.

\section{Proof of the main theorem} 
\label{sec:proof}

\noindent
We start by proving part (a) of Theorem \ref{t:main-theorem}.
 Let $X$ be an Oka-Stein manifold.  Let $K$ be a compact subset of $X$ and $m\geq 1$ be an integer.  Define $\mathscr{T}(m,K)\subset \End\,X$ to be the open set of endomorphisms of $X$ such that every periodic point $p\in K$ of minimal period at most $m$ is transverse.
 (A periodic point $p$ of an endomorphism $f$ of minimal period $n$ is transverse if $1$ is not an eigenvalue of the derivative $D_p f^n$.)  Note that we do not require the whole cycle of $p$ to be contained in $K$.  Define $\mathscr{H}(m,K)\subset \End\,X$ to be the open set of endomorphisms of $X$ such that every periodic point in $K$ of minimal period at most $m$ is hyperbolic. 

Here is the idea of the proof of the theorem.  We will show that for all compact subsets $K$ of $X$ and all $m\geq 1$, $\mathscr{H}(m,K)$ is dense in $\End\,X$.  The sets $\mathscr{T}(m,K)$ are more natural with respect to transversality theory and will help us prove the density of $\mathscr{H}(m,K)$.  Then we consider an exhaustion $(K_m)$ of $X$ and define
\[\mathscr{G}=\bigcap_{m\geq 1} \mathscr{H}(m,K_m).\]
The set $\mathscr G$ is residual in $\End\,X$ and every periodic point of an endomorphism in $\mathscr G$ is hyperbolic.

The density of $\mathscr{H}(m,K)$ is proved by induction.  The steps are the following.

\begin{proposition}
Let $K$ be a compact subset of an Oka-Stein manifold $X$.  Then $\mathscr{T}(1,K)$ is dense in $\End\,X$.
\end{proposition}

\begin{proof}
This is an immediate consequence of Thom's transversality theorem for holomorphic maps from a Stein manifold to an Oka manifold \cite[Corollary 8.8.7]{Forstneric2017}.
\end{proof}

\begin{proposition}
Let $K$ be a compact subset of an Oka-Stein manifold $X$.  Then $\mathscr{H}(m,K)$ is dense in $\mathscr{T}(m,K)$ for all $m\geq 1$.
\end{proposition}

\begin{proof}
Let $f\in\mathscr{T}(m,K)$.  Let $\Lambda$ be the set of periodic points of $f$ of minimal period at most $m$ in $K$.  It is finite, of cardinality, say, $N$.  By persistence of transverse periodic points, there is a neighbourhood $U$ of $\Lambda$ in $X$ and a neighbourhood $\mathscr N$ of $f$ in $\End\,X$ such that every $g\in\mathscr{N}$ has exactly $N$ periodic points of minimal period at most $m$ in $\overline U$.  Arguing as in \cite[Lemma 1]{AL2020}, we obtain a sequence $(g_n)$ in $\End\,X$ converging to $f$, such that for all $n$, every point in $\Lambda$ is a hyperbolic periodic point of minimal period at most $m$ for $g_n$.  Hence for $n$ large enough, all the periodic points of $g_n$ of minimal period at most $m$ in $\overline U$ are hyperbolic.  The result now follows from the fact that the non-existence of periodic points of minimal period at most $m$ in $K\setminus U$ is an open condition.
\end{proof}

\begin{proposition}
Let $K$ be a compact subset of an Oka-Stein manifold $X$.  Then, for all $m\geq 2$, $\mathscr{T}(m,K)\cap \mathscr{H}(m-1,K)$ is dense in $\mathscr{H}(m-1,K)$.
\end{proposition}

\begin{proof}
Let $f_0\in \mathscr{H}(m-1,K)$.  Recall that $ \mathscr{H}(m-1,K)$ is open.  We will obtain a perturbation of $f_0$, that is, a continuous map $f:P\to\End\,X$, where the parameter space $P$ is a neighbourhood of the origin $0$ in some $\C^k$, with $f(0)=f_0$,
such that there are arbitrarily small $t\in P$ with $f_t=f(t)\in \mathscr{T}(m,K)$.  Equivalently, we will show that there are arbitrarily small $t\in P$ such that the map 
\[\rho_m(f_t) : X\to X\times X,\quad  x\mapsto (x, f^m_t(x)),\]
is transverse to the diagonal $\Delta\subset X\times X$ on $K$.  Indeed, considering the derivative 
\[ d_a\rho_m(f_t) : T_aX\to T_aX\oplus T_aX, \quad v\mapsto (v,d_af_t^m),\]
we see that $d_a\rho_m(f_t)(T_aX)\cap \Delta=\{0\}$ if and only if $1$ is not an eigenvalue of $d_af_t^m$.  By the parametric transversality theorem \cite[Theorem 3.2.7]{Hirsch1976}, this is the case if the associated map
\begin{equation}\label{parametricperturbation}
F: X\times P\to X\times X,\quad F(x,t)=(x,f^m_t(x)),
\tag{*}\end{equation} 
is $C^1$ and transverse to the diagonal $\Delta$ on $K\times P$.

Consider the finite set $\Lambda$ of periodic points of $f_0$ of minimal period at most $m-1$ in $K$.  Let $U$ be a neighbourhood of $\Lambda$ and let $\mathscr{N}$ be a neighbourhood of $f_0$ in $\mathscr{H}(m-1,K)$.  By persistence of hyperbolic periodic points, we can choose $U$ and $\mathscr{N}$ small enough that no $g\in \mathscr{N}$ has periodic points of minimal period $m$ in $\overline U\cap K$.  Since $\mathscr{N}\subset \mathscr{H}(m-1,K)$, it follows that for all $g\in \mathscr{N}$, the map $\rho_m(g): x\mapsto (x, g^m(x))$ is transverse to $\Delta$ on $\overline U\cap K$. 

Take $x\in K\setminus U$ and let $V\subset X$ be a neighbourhood of $x$.  Since $x\not\in \Lambda$, the points $x, f_0(x), \ldots,f_0^{m-1}(x)$ are distinct.  Hence, by choosing $V$ small enough, we can ensure that for $j=1,\ldots, m-1$, the set $f_0^j(\overline V)$ is contained in a compact set $L_j$ 
such that
\begin{itemize}
\item  the sets $\overline V, L_1, \dots, L_{m-1}$ are mutually disjoint, and
\item the union $\overline V\cup  L_1\cup \cdots\cup L_{m-1}$ is $\mathscr{O}(X)$-convex. 
\end{itemize}
Let $b$ be a holomorphic function on $X$, close to 1 on $\overline V$ and close to 0 on $ L_1\cup  \dots\cup L_{m-1}$.  Since $K\setminus U$ is compact, there is a finite cover $\{V_1, \ldots, V_\ell$\} of $K\setminus U$ and associated holomorphic functions $b_1, \ldots, b_\ell$.  Let $s: X\times \C^r\to X$ be a dominating spray.  Consider the perturbations 
\[\varphi_t^{V_j}(x)= s(x, b_j(x)t), \quad j=1,\ldots, \ell.\]
Define the perturbation
\[f:X\times (\C^r)^\ell\to X, \quad f_t(x)=f(x,t_1,\ldots, t_\ell)=\varphi_{t_1}^{V_1}\circ\dots\circ \varphi_{t_\ell}^{V_\ell}\circ f_0(x).\]
To complete the proof we just need to verify that for $k=1,\ldots, \ell$, the associated map 
\[F:X\times \C^{r\ell}\to X\times X, \quad F(x,t)=(x, f_t^m(x)),\]
is transverse to $\Delta$ on $V_k\times \{0\}$.
Fix $k$ and let $a\in V_k$ have $f_0^m(a)=a$.  Set $t_j=0$ for $j\neq k$.  Fix $x=a$ and let $t_k$ vary.  It suffices to show that the map 
\[ \Theta: \C^r\to X, \quad t_k\mapsto (\varphi_{t_k}^{V_k}\circ f_0)^m(a), \]
is a submersion at the origin.  Denoting by $\tilde s: X\times \C^r\to X$ the modified spray with $\tilde s(x,t)=s(x,b_k(x)t)$, we see that the derivative $d_0\Theta: \C^r \to T_aX$ is given by
\[ d_0\Theta=\frac{\partial \tilde s}{\partial t}\bigg\vert_{(a,0)}+ d_{f_0^{m-1}(a)}f_0\circ \frac{\partial \tilde s}{\partial t}\bigg\vert_{(f_0^{m-1}(a),0)}+\cdots+ d_{f_0(a)}f_0^{m-1}\circ \frac{\partial \tilde s}{\partial t}\bigg\vert_{(f_0(a),0)}.\]
Choose a hermitian metric $h$ on $X$.  Let $C>0$ such that for $j=1,\ldots, m-1$ and for all $x\in \overline{f_0^j(V_k)}$, we have
$\lVert d_x f_0^{m-j}\rVert\leq C$ (using the operator norm associated to $h$).  Since $s$ is dominating, there is $\epsilon>0$ such that for all $x\in \overline {V_k}$, the ball $B\bigg(\dfrac{\partial  s}{\partial t}\bigg\vert_{(x,0)}, \epsilon\bigg)$ in the space of linear maps from $\C^k$ to $T_xX$ is contained in the open subset of maps of maximal rank.  We can choose $b_k$ close enough to 1 on $\overline{V_k}$ that for all $x\in\overline{V_k}$,
\[\frac{\partial \tilde s}{\partial t}\bigg\vert_{(x,0)}\in B\left(\frac{\partial  s}{\partial t}\bigg\vert_{(x,0)}, \frac{\epsilon}{2}\right),\]
and close enough to 0 on $f_0(\overline V_k)\cup\dots \cup f_0^{m-1}(\overline V_k)$ that for all $y\in f_0(\overline V_k)\cup\dots \cup f_0^{m-1}(\overline V_k)$, 
\[\left\|\frac{\partial \tilde s}{\partial t}\bigg\vert_{(y,0)}\right\|\leq \frac{\epsilon}{2C(m-1)}.\]
It follows that $d_0\Theta$ belongs to the ball $B\bigg(\dfrac{\partial  s}{\partial t}\bigg\vert_{(x,0)}, \epsilon\bigg)$ and thus has maximal rank.
\end{proof}

\begin{remark}
Let $f_0\in\End\,X$, let $K\subset X$ be compact, and fix $m\geq 1$.  One could try to avoid the induction process and find a perturbation $f: P\to \End\,X$ such that the associated map $F$ in \eqref{parametricperturbation} above is $C^1$ and transverse to the diagonal $\Delta$ on $K\times P$.   A strategy to obtain such a perturbation could be to show that if $a\in K$ satisfies $f_0^m(a)=a$, then the linear map 
\[  A=\frac{\partial f_t^m(x)}{\partial t}\bigg\vert_{(a,0)}: \C^r\to T_aX\]
is surjective.  There is, however, a simple obstruction to the surjectivity of $A$.  Assume that the minimal period $m_0$ of $a$ is smaller that $m$ and let $d=m/m_0$.  Assume moreover that $B=d_af_0^{m_0}$ admits as an eigenvalue a $d$-th root $\zeta\neq 1$ of unity.  Then $A$ is equal to 
\[ A=(B^{d-1}+B^{d-2}+\cdots +I)\circ T,\]
where $T:\C^r\to T_aX$ is a linear map.  The map $B^{d-1}+B^{d-2}+\cdots +I: T_aX\to T_aX$ is not surjective since $\zeta$ is an eigenvalue of $B$, so $A$ cannot be surjective.
 This is why, in the induction process, we take care of points with period smaller than $m$ first.
\end{remark}

\bigskip

A key ingredient in the proof of the remainder of Theorem \ref{t:main-theorem}, parts (b)--(f), is the following result, which builds on the proofs of the closing lemma and general density theorem in our previous paper \cite[Theorem 2(c) and Corollary 1(b)]{AL2020}.

\begin{theorem}   \label{t:key-theorem}
For a generic endomorphism $f$ of an Oka-Stein manifold $X$, the set $X\setminus \rne(f)$ lies in the closure of $\att(f)$, equals the closure of $\rep(f)$, and, if $\dim X\geq 2$, equals the closure of $\sad(f)$.
\end{theorem}

To prove the theorem we use the following closing lemma.  The term non-degenerate in its proof refers to an endomorphism that has maximal rank at some point or, in other words, is a local biholomorphism outside a proper subvariety.

\begin{theorem}   \label{t:key-closing-lemma}
Let $f$ be an endomorphism of an Oka-Stein manifold $X$, let $K\subset X$ be compact, and let $p\in X\setminus \rne(f)$.  Let $W$ be a neighbourhood of $f$ in $\End\,X$ and $V$ be a neighbourhood of $p$ in $X$.  Then:
\begin{enumerate}
\item[(a)]  There are $h\in W$ and $q\in V$ such that $q$ is a super-attracting periodic point of $h$ and the $h$-orbit of $q$ leaves $K$.
\item[(b)]  There are $h\in W$ and $q\in V$ such that $q$ is a repelling periodic point of $h$ and the $h$-orbit of $q$ leaves $K$.
\item[(c)]  Suppose that $\dim X\geq 2$.  There are $h\in W$ and $q\in V$ such that $q$ is a saddle periodic point of $h$ and the $h$-orbit of $q$ leaves $K$.
\end{enumerate}
\end{theorem}

\begin{proof}
Take a holomorphically convex compact subset $L$ of $X$ containing $K$.  Since $p\in X\setminus \rne(f)$, there are $g\in W$ and $q\in V$ such that the $g$-orbit of $q$ is not contained in $L$.  Say $g^k(q)\in L$ for $0\leq k<m$ and $g^m(q)\notin L$.  By Proposition \ref{p:max-rank-prop} below, we may assume that $g$ is non-degenerate, and by Lemma \ref{l:max-rank-lemma} below, we may assume that $g^m$ has maximal rank at $q$.

Let $\phi:X\to X$ be continuous, equal to $g$ on a neighbourhood of $L$, and holomorphic on a neighbourhood of $g^m(q)$ with $\phi(g^m(q))=q$ and an arbitrarily prescribed derivative at $g^m(q)$.  Since $X$ is Stein and Oka, $\phi$ can be deformed to $h\in\End\,X$, arbitrarily close to $g$ on $L$, such that the 1-jet of $h$ coincides with that of $g$ at the points $q, g(q), \ldots, g^{m-1}(q)$, and such that $h$ takes $g^m(q)$ to $q$ with the prescribed derivative at $g^m(q)$.  Then $h$ is arbitrarily close to $g$ on $L$ with $q$ as a periodic point and $h^m(q)\notin K$.  Now
\[ (h^{m+1})'(q) = h'(g^m(q)) (g^m)'(q),\]
so we can get $q$ to be super-attracting by taking $h'(g^m(q))=0$.  Since $(g^m)'(q)$ is non-singular, we can get $q$ to be a repelling periodic point or, when $\dim X\geq 2$, a saddle periodic point by a suitable choice of $h'(g^m(q))$.
\end{proof}

\begin{lemma}  \label{l:max-rank-lemma}
Let $g$ be a non-degenerate endomorphism of a complex manifold $X$.  Let $m\geq 1$.  Then there is a dense open subset $V$ of $X$ such that the $m^\textrm{th}$ iterate $g^m$ has maximal rank at every point of $V$.
\end{lemma}

\begin{proof}
Let $U$ be an open dense subset of $X$.  Then the open set $g^{-1}(U)$ is also dense.  Indeed, suppose that there exists a nonempty open subset $W$ such that $g(W)\cap U =\varnothing$.  Since $g$ has maximal rank on a dense set, there are points in $W$ at which $g$ has maximal rank, so $g(W)$ has nonempty interior, which contradicts $U$ being dense.

Now let $U$ be the set of points where $g$ has maximal rank.  Then $V=U\cap g^{-1}(U)\cap\cdots\cap g^{1-m}(U)$ is open and dense and $g^m$ has maximal rank at every point of $V$.
\end{proof}

\begin{proposition}  \label{p:max-rank-prop}
Let $X$ be an Oka-Stein manifold and $p\in X$.  A generic endomorphism of $X$ is a local biholomorphism at $p$.  Hence, a generic endomorphism of $X$ is non-degenerate.
\end{proposition}

\begin{proof}
The set of endomorphisms of maximal rank at $p$ is open.  We need to show that it is dense.  Embed $X$ in $\C^m$ with a tubular neighbourhood that retracts holomorphically onto $X$.  Take $p$ to be the origin in $\C^m$ and suppose that $f\in\End\,X$ is singular at $0$ with $f(0)=0$.  Add to $f$ the map $\epsilon\chi\lambda$, where $\lambda$ is a linear automorphism of $\C^m$, $\epsilon>0$, and the continuous function $\chi:X\to[0,1]$ is $1$ on a large compact subset of $X$ and zero near infinity.  If $\epsilon$ is small enough, the image of the new map lies in the tubular neighbourhood, and by a suitable generic choice of $\lambda$, the new map composed by the retraction has maximal rank at $0$.  Finally apply the basic Oka property with approximation and jet interpolation and obtain an endomorphism close to $f$ and of maximal rank at $0$.
\end{proof}

\begin{proof}[Proof of Theorem \ref{t:key-theorem}]
Let $\{U_n:n\geq 1\}$ be a countable basis for the topology of $X$.  Let $S_n$ be the open set of all $f\in\End\,X$ such that $f$ has an attracting cycle intersecting $U_n$.  Then $G=\bigcap \End\,X\setminus \partial S_n$ is a residual subset of $\End\,X$.  We will show that if $f\in G$ and $p\in X\setminus\rne(f)$, then $p$ lies in the closure of $\att(f)$.

Suppose that $p\in U_n$.  It suffices to show that $\att(f)\cap U_n\neq\varnothing$.  By definition of $G$, we have $f\notin\partial S_n$.  Hence either $f\in S_n$, in which case $\att(f)\cap U_n\neq\varnothing$ is immediate, or $f\notin\overline S_n$.  The latter case is ruled out by Theorem \ref{t:key-closing-lemma}.  Indeed, by the theorem, if $p\in U_n\setminus\rne(f)$, then there is an endomorphism arbitrarily close to $f$ with a super-attracting periodic point in $U_n$.

The proofs for $\rep(f)$ and $\sad(f)$ are analogous.  It easily follows from Cauchy estimates, as in \cite[Claim~3]{AL2019}, that $\rep(f)$ and $\sad(f)$ are subsets of $X\setminus\rne(f)$.
\end{proof}

\begin{lemma}   \label{l:basins}
Let $f$ be an endomorphism of a complex manifold $X$.
\begin{enumerate}
\item[(a)]  The basin of attraction of an attracting cycle of $f$ is contained in $\rne(f)$ \ldots
\item[(b)]  \ldots and is a union of Fatou components.
\item[(c)]  $\att(f)$ is a closed subset of $F_f$ and hence, if $X$ is Stein, of $\rne(f)$.
\end{enumerate}
\end{lemma}

Let us review the notion of the basin of attraction of an attracting cycle.  It is not quite as straightforward as the special case of an attracting fixed point.  Consider an attracting cycle $A$ of $f$ of length $m$.  Every point of $A$ is an attracting fixed point of $f^m$.  There are at least four ways to define the basin of attraction of $A$.

Say $x\in B_1$ if $(f^{mn}(x))$ converges to a point in $A$ as $n\to\infty$.  Say $x\in B_1'$ if $(f^{mn})$ converges uniformly to a point in $A$ on a neighbourhood of $x$ as $n\to\infty$.  Then $B_1=B_1'$ is the union of the basins of attraction of the points of $A$ viewed as fixed points of $f^m$.  Say $x\in B_2$ if $(f^n(x))$ converges to $A$ as $n\to\infty$, meaning that for every neighbourhood $V$ of $A$, $f^n(x)\in V$ for all sufficiently large $n$.  Finally say $x\in B_2'$ if $(f^n)$ converges uniformly to $A$ on a neighbourhood $U$ of $x$ as $n\to\infty$, meaning that for every neighbourhood $V$ of $A$, $f^n(U)\subset V$ for all sufficiently large $n$.  Clearly, $B_2'\subset B_2\subset B_1$.

It is now a basic observation that $f^n\to A$ uniformly on a neighbourhood of $A$.  In other words, $A\subset B_2'$.  It follows that $B_1\subset B_2'$, so $B_1=B_1'=B_2=B_2'$.  Thus there is a single natural notion of the basin $B$ of $A$.  The basin is open and completely invariant, that is, $f^{-1}(B)=B$.  Also, the basin is a subset of the Fatou set $F_f$.  Indeed, if $x\in B$, given any sequence $(f^{n_k})$ of iterates of $f$, there are infinitely many $k$ for which $n_k$ lies in the same congruence class modulo $m$, so we can extract a subsequence of $(f^{n_k})$ that converges uniformly in a neighbourhood of $x$ to a point in the cycle. 

\begin{proof}
(a)  Another basic observation is that there is a compact neighbourhood $K$ of $A$ such that $f(K)\subset K^\circ$ (in fact an arbitrarily small one, but this is not needed here).  Take $p\in B$ and find a compact neighbourhood $L$ of $p$ with $f^n(L)\subset K^\circ$ for some $n\geq 0$.  Then, for all $g\in\End\,X$ sufficiently close to $f$, $g(K)\subset K^\circ$ and $g^n(L)\subset K^\circ$, so $g^j(L)\subset K$ for all $j\geq n$.  It follows that $p\in\rne(f)$.

(b)  Let us prove that if a Fatou component $U$ intersects $B$, then  $U\subset B$.  Take $p\in U\cap B$.  There is a neighbourhood $V$ of $p$ such that $f^{mn}$ converges uniformly on $V$ to a point $q$ in the cycle as $n\to\infty$.  Since $U$ is a Fatou component, and no subsequence of $(f^{mn})$ converges to infinity locally uniformly on $U$, every subsequence of $(f^{mn})$ has a subsequence that converges locally uniformly on $U$ to a holomorphic map $g:U\to X$.  Since $g$ is constantly equal to $q$ on $V$, it is constantly equal to $q$ on all of $U$.  Hence, $f^{mn}$ converges to $q$ locally uniformly on $U$.

(c)  Let $(x_n)$ be a sequence in $\att(f)$ converging to a point $y$ in $F_f$.  The Fatou component containing $y$ can contain at most one point of $\att(f)$.  Hence the points $x_n$ are eventually equal to $y$.
\end{proof}

\begin{proposition}   \label{p:rne-equals-Fatou}
For a generic endomorphism $f$ of an Oka-Stein manifold $X$, 
\[ F_f=\rne(f). \]
Moreover, $F_f$ is the set of points for which there is a neighbourhood $V$ and a compact subset $K$ of $X$ with $f^n(V)\subset K$ for all $n\geq 1$.
\end{proposition}

\begin{proof}
Take $p\in F_f\setminus\rne(f)$ and let $U$ be the component of $F_f$ containing $p$.  By Lemma \ref{l:basins}(a), $U$ is not a component of a basin of attraction; yet, by Theorem \ref{t:key-theorem}, there are attracting periodic points arbitrarily close to $p$ and hence in $U$, which is absurd by Lemma \ref{l:basins}(b).

Let $E$ be the set of points for which there is a neighbourhood $V$ and a compact subset $K$ of $X$ with $f^n(V)\subset K$ for all $n\geq 1$.  Then $\rne(f)\subset E\subset F_f$ and the second claim follows.
\end{proof}

The following lemma is Forn\ae ss and Sibony's \cite[Proposition 2.5]{FS1998}.  They proved it for affine space, but the proof easily extends to Stein manifolds.

\begin{lemma}   \label{l:Fatou-component}
Let $f$ be an endomorphism of a Stein manifold and let $U$ be a Fatou component of $f$.  If $(f^n)$ is locally bounded on $U$, then $U$ is Runge (in particular Stein).  If, moreover, $f$ is non-degenerate, then $f(U)\subset F_f$.
\end{lemma}

\begin{proposition}   \label{p:invariance}
The following hold for a generic endomorphism $f$ of an Oka-Stein manifold $X$.
\begin{enumerate}
\item[(a)]  The Fatou set $F_f$ is forward invariant and the components of $F_f$ are Runge.
\item[(b)]  Both $F_f$ and $J_f$ are completely invariant.
\end{enumerate}
\end{proposition}

\begin{proof}
(a) follows from Propositions \ref{p:max-rank-prop} and \ref{p:rne-equals-Fatou} and Lemma \ref{l:Fatou-component}.

(c) follows immediately from (a) and the easy observation that $F_f$ is backward invariant.
\end{proof}

\begin{proposition}   \label{p:F-and-J-of-iterates}
The following hold for a generic endomorphism $f$ of an Oka-Stein manifold $X$.
\begin{enumerate}
\item[(a)]  $J_f=\overline{\rep(f)}$ and, if $\dim X\geq 2$, $J_f=\overline{\sad(f)}$.
\item[(b)]  For every $m\geq 2$,
\[ J_{f^m}=J_f \qquad\textrm{and}\qquad F_{f^m}=F_f. \]
\end{enumerate}
\end{proposition}

\begin{proof}
(a) follows from Theorem \ref{t:key-theorem} and Proposition \ref{p:rne-equals-Fatou}.

(b)  Clearly, $F_f\subset F_{f^m}$, so $J_{f^m}\subset J_f$.  Since $J_f$ is the closure of $\rep(f)$, which is contained in $J_{f^m}$, the reverse inclusions hold.
\end{proof}

From Theorem \ref{t:key-theorem}, Lemma \ref{l:basins}(c), and Proposition \ref{p:rne-equals-Fatou}, we deduce the following result.

\begin{proposition}   \label{p:big-proposition}
For a generic endomorphism $f$ of an Oka-Stein manifold $X$,
\[ \clo(f) = \Omega_f = \overline{\att(f)}=J_f\cup\att(f). \]
Also, $J_f=X\setminus\rne(f)$ has no interior.
\end{proposition}

\begin{proof}
Since $\att(f)$ is closed in $F_f$ and $J_f$ lies in the closure of $\att(f)$, we have $\overline{\att(f)} = J_f\cup\att(f)$.  Since also $\att(f)\cap J_f=\varnothing$, the interior of $J_f$ is empty.

Now the non-wandering set $\Omega_f$ is closed and $\att(f)\subset\Omega_f$, so $\overline{\att(f)} \subset \Omega_f$.  The closing lemma \cite[Theorem~2(c)]{AL2020} says that $\Omega_f \subset \clo(f)$.  The proof of the general density theorem \cite[Corollary~1]{AL2020} shows that $\hyp(f)$ is dense in $\clo(f)$ for a generic endomorphism $f$, so $\clo(f)=\overline{\hyp(f)}=\att(f)\cup J_f$.  Thus the four sets $\overline{\att(f)} \subset \Omega_f \subset \clo(f) =J_f\cup \att(f)$ are equal.
\end{proof}

\begin{proposition}   \label{p:more-about-Julia}
The following hold for a generic endomorphism $f$ of an Oka-Stein manifold $X$.
\begin{enumerate}
\item[(a)]  For every compact subset $K$ of $X$, $f$ has an attracting fixed point, a saddle fixed point, and a repelling fixed point outside $K$.
\item[(b)]  $J_f$ is not compact.
\item[(c)]  $J_f$ is perfect.
\end{enumerate}
\end{proposition}

\begin{proof}
(a)  Let $K\subset X$ be compact.  We will show that the open set of endomorphisms of $X$ with an attracting fixed point outside $K$ is dense in $\End\,X$, and similarly for saddle fixed points and repelling fixed points.  

Let $f\in\End\,X$, let $L$ be a holomorphically convex compact subset of $X$ containing $K$, and let $p\in X\setminus L$.  Let $\phi:X\to X$ be continuous, equal to $f$ on a neighbourhood of $L$, and holomorphic on a neighbourhood of $p$ with $\phi(p)=p$ and an arbitrarily prescribed derivative at $p$.  Since $X$ is Stein and Oka, $\phi$ can be deformed to an endomorphism of $X$, arbitrarily close to $f$ on $L$, whose 1-jet at $p$ coincides with that of $\phi$.

(b) follows from (a).

(c)  Suppose that $p$ is an isolated point of $J_f$.  Let $U$ be a coordinate ball centred at $p$ such that $U\setminus\{p\}\subset F_f$.  By Proposition \ref{p:big-proposition}, $p\in\overline{\att(f)}$, so there are two points in $\att(f)\cap U$ that lie in distinct cycles.  Since $U\setminus\{p\}$ is connected, the two points must lie in the same Fatou component, which is absurd.
\end{proof}

\begin{proposition}   \label{p:chaos}
A generic endomorphism of an Oka-Stein manifold is chaotic on its Julia set.
\end{proposition}

The usual definition of chaos is that periodic points are dense and there is a dense orbit.  Instead, we verify Touhey's characterisation of chaos \cite{Touhey1997}, which says that for every two nonempty open subsets, there is a cycle that visits both of them.

\begin{proof}
Let $\Lambda\subset \End\,X$ be the dense open set of non-degenerate endomorphisms.  Take a countable basis $\{U_n:n\geq 1\}$ for the topology of $X$.  Let $S_{m,n}$ be the open subset of $\End\,X$ of endomorphisms with a repelling cycle through $U_m$ and $U_n$.  Then 
\[ G=\Lambda\cap \bigcap\limits_{m,n} \End\,X \setminus \partial S_{m,n}\] 
is residual.  We claim that each $f\in G$ is chaotic on $J_f$.

Take $U_m$ and $U_n$ both intersecting $J_f$.  By definition of $G$, we have $f\in S_{m,n}$ or $f\notin \overline S_{m,n}$.  In the former case, $f$ has a repelling cycle through $U_m$ and $U_n$.  Such a cycle lies in $J_f$ and we are done.  We rule out the latter case by showing that $f$ can be approximated by endomorphisms in $S_{m,n}$.  

Let $K$ be a holomorphically convex compact subset of $X$.  There are $p\in U_m$ and $q\in U_n$ such that $f^j(p)$ and $f^k(q)$ lie outside $K$ for some $j,k\geq 1$.  Since $f\in \Lambda$, by Lemma \ref{l:max-rank-lemma}, we can assume that $f^j$ has maximal rank at $p$ and that $f^k$ has maximal rank at $q$.  Let $\phi:X\to X$ be continuous, equal to $f$ on a neighbourhood of $K$ and at the points $p,f(p),\ldots,f^{j-1}(p)$ and $q,f(q),\ldots,f^{k-1}(q)$, holomorphic on neighbourhoods of $f^j(p)$ and $f^k(q)$, with $\phi(f^j(p))=q$ and $\phi(f^k(q))=p$ and an arbitrarily prescribed derivative of maximal rank at those points.  Deform $\phi$ to an endomorphism $g$ of $X$, arbitrarily close to $f$ on $K$, equal to $f$ at the points $p,f(p),\ldots,f^{j-1}(p)$ and $q,f(q),\ldots,f^{k-1}(q)$, and with the prescribed 1-jet at $f^j(p)$, $f^k(q)$.  By taking the derivatives $\phi'(f^j(p))$ and $\phi'(f^k(q))$ to be sufficiently large, we get $g\in S_{m,n}$.
\end{proof}

\begin{proposition}   \label{p:bounded-orbits}
The following hold for a generic endomorphism $f$ of an Oka-Stein manifold $X$.
\begin{enumerate}
\item[(a)]  The set of points in $J_f$ whose $f$-orbit is relatively compact is dense in $J_f$.
\item[(b)]  The set of points in $J_f$ whose $f$-orbit is not relatively compact is dense in $J_f$.
\item[(c)]  $J_f$ is the boundary of the set of points in $X$ whose $f$-orbit is relatively compact.
\end{enumerate}
\end{proposition}

\begin{proof}
(a)  By Proposition \ref{p:chaos}, points with finite orbit are dense in $J_f$.

(b)  By Proposition \ref{p:chaos}, points with dense orbit are dense in $J_f$, so by Proposition \ref{p:more-about-Julia}(b), points whose orbit is not relatively compact are dense in $J_f$.

(c)  follows from (b) and Proposition \ref{p:rne-equals-Fatou}.
\end{proof}

\begin{proposition}   \label{p:Julia-images}
The following hold for a generic endomorphism $f$ of an Oka-Stein manifold $X$.
\begin{enumerate}
\item[(a)]  If $U$ is a neighbourhood of a point in $J_f$, then $\bigcup\limits_{n\geq 0}f^n(U)$ is dense in $X$.
\item[(b)]  $J_f$ is the boundary of each basin of attraction.
\end{enumerate}
\end{proposition}

\begin{proof}
(a)  Let $\{U_n:n\geq 1\}$ be a countable basis for the topology of $X$.  Let $S_{m,n}$ be the open set of all $f\in\End\,X$ such that $f^k(U_m)$ intersects $U_n$ for some $k\geq 0$.  Let $G=\bigcap \End\,X\setminus \partial S_{m,n}$.  Then $G$ is a residual subset of $\End\,X$.  We prove (a) by showing that for every $f\in G$ and every $m,n\geq 1$ such that $U_m$ intersects $J_f$, there is $k\geq 0$ such that $f^k(U_m)$ intersects $U_n$.  If $f\in S_{m,n}$, this is obvious.  By the definition of $G$, if $f\notin S_{m,n}$, then $f\notin \overline S_{m,n}$, but this can be ruled out as follows.

Let $K$ be a holomorphically convex compact subset of $X$.  Since $U_m$ intersects $J_f$, there is $p\in U_m$ such that $f^k(p)\notin K$ for some $k\geq 1$.  Let $\phi:X\to X$ be continuous, equal to $f$ on a neighbourhood of $K$ and at the points $p,f(p),\ldots,f^{k-1}(p)$, and with $\phi(f^k(p))\in U_n$.  Deform $\phi$ to an endomorphism $g$ of $X$, arbitrarily close to $f$ on $K$ and equal to $\phi$ at the points $p,f(p),\ldots,f^k(p)$.  Then $g\in S_{m,n}$ approximates $f$ as well as desired.

(b)  Let $B$ be the basin of attraction of an attracting cycle of $f$ and let $U$ be a neighbourhood of a point in $J_f$.  By (a), $f^n(U)$ intersects $B$ for some $n\geq 0$, so $U$ itself intersects $B$.  Hence, $J_f\subset\partial B$.  The opposite inclusion is evident.
\end{proof}

\begin{proposition}  \label{p:escaping}
Let $f$ be an endomorphism of an Oka-Stein manifold $X$ and let $p\in X\setminus \rne(f)$.  Take a neighbourhood $U$ of $p$ in $X$ and $V$ of $f$ in $\End\, X$.  Then there is an endomorphism in $V$ with an escaping point in $U$. 
\end{proposition}

\begin{proof}
Choose for convenience an embedding of $X$ in some affine space and let $\lVert\cdot\rVert$ be the sup-norm there.  Take a holomorphically convex compact subset $K$ of $X$ and $\epsilon>0$ such that $\{g\in \End\,X : \lVert f-g\rVert_K\leq 2\epsilon\}\subset V$.  Since $p\in X\setminus \rne(f)$, there is $q\in U$ and $h_0\in \End\,X $ such that $\|f-h_0\|_K<\epsilon$ and $h_0^m(q)\not\in K$ for some $m\geq 0$.  Let $q_j=h_0^j(q)$, $j=0,\ldots,m$.

Exhaust $X$ by holomorphically convex compact subsets $K_n$, $n\geq 0$, such that $K_0=K$.  We may assume that $q_m\in K_1\setminus K_0$.  Applying the basic Oka property with approximation and interpolation, we obtain $h_1\in \End\,X $ such that $\lVert h_1-h_0\rVert_{K_0}<\epsilon/2$, $h_1(q_j)=q_{j+1}$ for $j=0,\ldots,m-1$, and $h_1(q_m)\in K_2\setminus K_1$.  Continuing in this way, we obtain a sequence $(h_n)$ in $\End\,X $ such that for all $n\geq 1$,
\[ \lVert h_{n+1}-h_n\rVert_{K_n}<\frac{\epsilon}{2^{n+1}}, \] 
\[ h_n^{j}(q)=h_{n-1}^{j}(q),\quad \textrm{for } j=0,\ldots,m+n-1, \]
and
\[ h_n^{m+n}(q) \in K_{n+1}\setminus K_n. \]
Hence, $(h_n)$ converges to an endomorphism in $V$ for which $q$ is escaping.
\end{proof}

The final claims in parts (b) and (c) of Theorem \ref{t:main-theorem} now follow from Propositions \ref{p:rne-equals-Fatou} and \ref{p:escaping} and the equivalence of transitivity and micro-transitivity for a continuous action of a Polish group on a Polish space (theorem of Effros; see the introduction to \cite{AL2020} and the references there).  

The proof of Theorem \ref{t:main-theorem} is complete.

\section{Fatou components of generic endomorphisms} 
\label{sec:Fatou}

\noindent
Let us first review a few more definitions.  A Fatou component $U$ of an endomorphism $f$ of a complex manifold is wandering if the sets $f^n(U)$, $n\geq 0$, are mutually disjoint.  The Fatou component $U$ is recurrent if for some point $p$ in $U$, a subsequence of $(f^n(p))$ converges to a point $q$ in $U$ (so $q$ is an $\omega$-limit point of $p$).  Let us call $U$ {\it pre-recurrent} (this is not standard terminology) if some point in $U$ has an $\omega$-limit point $q$ in the Fatou set of $f$.  Being an $\omega$-limit point, $q$ is non-wandering.  Clearly, a Fatou component in the basin of attraction of an attracting cycle is pre-recurrent.

The Fatou component $U$ is non-wandering if there are $m < n$ such that $f^m(U)$ and $f^n(U)$ intersect.  Write $n=m+k$, $k\geq 1$.  If the manifold is Oka-Stein and $f$ is generic, then the Fatou set of $f$ is forward invariant by Theorem \ref{t:main-theorem}; assume this.  Then $f^m(U)$ lies in a Fatou component $V$, and $V$ and $f^k(V)$ intersect, so $f^k(V)\subset V$.  Thus $V$ is periodic and $U$ is preperiodic.  (In our usage, a periodic component is preperiodic and a recurrent component is pre-recurrent.)

\begin{remark}  \label{r:pre-recurrent}
As mentioned above, if $f$ is an endomorphism of an Oka-Stein manifold, then $\rne(f)\cap\Omega_f = \att(f)$, and by Theorem \ref{t:main-theorem}(b), if $f$ is generic, then $\rne(f)=F_f$.  Hence, by Lemma \ref{l:basins}(b), every pre-recurrent Fatou component of a generic endomorphism lies in the basin of attraction of an attracting cycle.
\end{remark}

In the case of the complex plane $\C$, our results show that a generic endomorphism does not have parabolic cycles, Siegel discs, or Baker domains.  It only has Fatou components of one type -- or perhaps two, if dynamically bounded wandering domains exist, which is an open question.  (For dynamics of entire functions, see the survey \cite{Schleicher2010}.)

\begin{theorem}   \label{t:entire}
A point in the Fatou set of a generic endomorphism $f$ of $\C$ is attracted to an attracting cycle or lies in a dynamically bounded wandering domain.  Every Fatou component of $f$ is a disc.
\end{theorem}

Hoping that it may be of interest to the reader, we give a proof that does not use the classification of Fatou components for entire functions.

\begin{proof}
First note that transcendental functions are generic among all entire functions.  On a multiply-connected Fatou component of a transcendental function, the iterates of the function converge to infinity (this is noted by Baker at the beginning of the proof of Theorem 1 in \cite{Baker1975}; see also \cite[Theorem 2.5]{Schleicher2010}).  By Theorem \ref{t:main-theorem}, points in the Fatou set of a generic endomorphism $f$ of $\C$ have relatively compact orbits and the Fatou set is not all of $\C$.  Hence, the Fatou components of $f$ are discs.  Also, a wandering domain of $f$ cannot be escaping or oscillating: it must be dynamically bounded.

Now let $U$ be a periodic component of $F_f=\rne(f)$ of period $m$.  If $U$ is recurrent, then $U$ lies in the basin of attraction of an attracting cycle by Remark \ref{r:pre-recurrent}.  Next assume that $U$ is not recurrent and take a sequence of iterates of $f^m$ converging locally uniformly on $U$ to a limit function $g$ with image in $J_f$.  Since $J_f$ has no interior by Theorem \ref{t:main-theorem}, $g$ is constant, and its value $p$ is a fixed point of $f^m$.  By Theorem \ref{t:main-theorem} once again, $p$ is hyperbolic and hence repelling, which is absurd.
\end{proof}

Motivated by Theorem \ref{t:entire}, we pose the following open question.

\smallskip\noindent
{\bf Question.}  Are all the Fatou components of a generic endomorphism $f$ of an Oka-Stein manifold pre-recurrent?  Equivalently (by Remark \ref{r:pre-recurrent}), is the Fatou set of $f$ the union of the basins of attraction of the attracting cycles of $f$?

\smallskip
We will not hazard a conjecture as to whether the answer is affirmative or negative.  Theorem \ref{t:main-theorem}(a) shows that arbitrarily small perturbations of an endomorphism can destroy all non-hyperbolic periodic points at once.  An affirmative answer would have a similar but apparently much stronger consequence, namely that all non-pre-recurrent Fatou components can be destroyed at once by arbitrarily small perturbations.

We begin to shed some light on the question with the following result.  In the next section, we reformulate the question in several different ways using the notion of chain-recurrence.

\begin{theorem}   \label{t:non-recurrent-dense}
Let $X$ be an Oka-Stein manifold.  The set of endomorphisms with a non-pre-recurrent Fatou component is dense in $\End\,X$.
\end{theorem}

It follows that the property of the entire Fatou set consisting of basins of attraction can be destroyed by an arbitrarily small perturbation.

To produce non-pre-recurrent Fatou components, we use a theorem of
Hakim \cite[Theorem 1.6]{Hakim1997} (see also \cite{Abate2015} and \cite{AR2014}).  She proved that for every endomorphism $f$ of a neighbourhood of the origin $0$ in $\C^n$, fixing $0$ with a suitable 2-jet there, for example $f(z_1,\ldots,z_n)=(z_1+z_1^2, \ldots, z_n+z_n^2)$, there is a bounded connected open set $U$ in $\C^n$ with $0\in\partial U$ and $f(U)\subset U$, such that $f^k(z)\to 0$ as $k\to\infty$ for every $z\in U$, but the family of iterates of $f$ is not normal on any neighbourhood of $0$.  It follows that an endomorphism of a complex manifold with a suitable 2-jet at a fixed point $p$ has a non-pre-recurrent Fatou component with $p$ in its boundary.

\begin{proof}
Let $f\in\End\, X$.  By Hakim's theorem, it suffices to show that by an arbitrarily small perturbation of $f$ we can obtain an endomorphism $g$ with a fixed point at which $g$ has a prescribed 2-jet.

Let $K$ be a holomorphically convex compact subset of $X$.  Take $p\in X\setminus K$.  Let $\phi:X\to X$ be continuous, equal to $f$ on a neighbourhood of $K$, holomorphic on a neighbourhood of $p$, and with any prescribed 2-jet at $p$.  Since $X$ is Stein and Oka, the basic Oka property with approximation and jet interpolation holds for maps $X\to X$, so we can deform $\phi$ to an endomorphism $g$ of $X$, arbitrarily close to $f$ on $K$, and with the prescribed 2-jet at $p$.
\end{proof}

\section{Holomorphic Conley theory} 
\label{sec:Conley}

\noindent
In the final section, we bring together our results on generic dynamics on Oka-Stein manifolds and Conley's fundamental theorems on general dynamical systems \cite{Conley1978} as adapted to the noncompact setting by Hurley \cite{Hurley1992}.  We start by summarising the necessary definitions and stating the two fundamental theorems.

Let $X$ be a locally compact second countable metric space and let $f:X\to X$ be continuous.  Choose a metric $d$ on $X$ compatible with the topology of $X$.  Let $\epsilon: X\to(0,\infty)$ be continuous.  A finite sequence $x_0, x_1,\ldots, x_n$, $n\geq 1$, of points in $X$ is an $\epsilon$-chain or $\epsilon$-pseudo-orbit of length $n$ if $d(f(x_j), x_{j+1})<\epsilon (f(x_j))$ for $j=0,\ldots,n-1$.  A point $p$ in $X$ is chain-recurrent for $f$ if for every function $\epsilon$, there is an $\epsilon$-chain that begins and ends at $p$.  We denote by $C_f$ the set of chain-recurrent points of $f$.  Clearly, a non-wandering point is chain-recurrent.

A nonempty open subset $U$ of $X$ is absorbing for $f$ if $\overline{f(U)}\subset U$.  The closed set $A=\bigcap\limits_{n\geq 0}\overline{f^n(U)}$ is the attractor determined by $U$ (this is a decreasing intersection).  The basin of $A$ relative to $U$ is the open set $B(A,U)=\bigcup\limits_{n\geq 0}f^{-n}(U)\supset A$.  Finally, the basin of $A$ is the open set $B(A)=\bigcup B(A,U)$, where $U$ ranges over the absorbing sets that determine $A$.

The first fundamental theorem \cite[Theorem 1]{Hurley1992} states that
\[ X\setminus C_f = \bigcup\limits_A B(A)\setminus A, \]
where the union is taken over all the attractors $A$ of $f$.  The right-hand side of the equation is independent of the choice of the metric $d$ and invariant under topological conjugation, so $C_f$ is as well.  The equation also shows that $C_f$ is closed.

An equivalence relation is defined on $C_f$ by declaring points $p$ and $q$ equivalent if for every continuous $\epsilon:X\to (0,\infty)$, there is an $\epsilon$-chain from $p$ to $q$ and an $\epsilon$-chain from $q$ to $p$.  The equivalence classes are called chain-recurrence classes.

The second fundamental theorem \cite[Theorem 2]{Hurley1992} states that there is a complete Lyapunov function $L:X\to\mathbb R$ for $f$.  It has the following properties.
\begin{enumerate}
\item  $L$ is continuous.
\item  $L(f(x))\leq L(x)$ for all $x\in X$, with equality if and only if $x\in C_f$.
\item  $L$ is constant on each chain-recurrence class and takes different values on different classes.
\item  If $C$ and $C'$ are distinct classes such that for each $\epsilon$ there is an $\epsilon$-chain from $C$ to $C'$, then $L(C) > L(C')$.
\item  $L(C_f)$ is nowhere dense in $\mathbb R$.
\end{enumerate}
We can think of $L$ as denoting \lq\lq height\rq\rq, so that points in $X\setminus C_f$ move down with time, but points in $C_f$ stay at the same height.

We now specialise to the case of $X$ being an Oka-Stein manifold and $f:X\to X$ being holomorphic.

\begin{theorem}   \label{t:Conley-basics}
A generic endomorphism $f$ of an Oka-Stein manifold $X$ has the following properties.
\begin{enumerate}
\item[(a)]  The Julia set $J_f$ lies in a single chain-recurrence class.
\item[(b)]  Each attracting cycle is an attractor and a chain-recurrence class.
\item[(c)]  The basin of attraction of an attracting cycle contains no chain-recurrent points except the attracting cycle itself.
\item[(d)]  Every point in a non-pre-recurrent Fatou component is chain-recurrent and lies in the same chain-recurrence class as the Julia set.
\item[(e)]  The union of the non-pre-recurrent Fatou components equals the set of wandering chain-recurrent points.
\end{enumerate}
\end{theorem}

\begin{proof}
(a) follows from the Julia set containing a dense orbit by Theorem \ref{t:main-theorem}(e).

(b)  A sufficiently small neighbourhood of an attracting cycle is absorbing.  The attracting cycle is the attractor determined by any such neighbourhood.  It is easily seen that one cannot escape from an attracting cycle along an $\epsilon$-pseudo-orbit if $\epsilon$ is small enough.

(c)  If $V$ is the basin of attraction of an attracting cycle $A$, then clearly $V\subset \bigcup f^{-n}(U)$, where $U$ is a neighbourhood of $A$ as in the proof of (b), so $V\subset B(A)$ and $V\setminus A\subset B(A)\setminus A\subset X\setminus C_f$.

(d)  Take a point $p$ in a non-pre-recurrent Fatou component.  The orbit of $p$ is relatively compact by Theorem \ref{t:main-theorem}(b), so $p$ has an $\omega$-limit point $q\in J_f$.  By Theorem \ref{t:main-theorem}(e), arbitrarily close to $q$ is a point $r$ whose orbit comes arbitrarily close to $p$.  Let us travel from $p$ along its orbit, jump to $q$ when we get close to it, jump to $r$, travel along its orbit, and jump to $p$ when we get close to it.  This makes an $\epsilon$-pseudo-orbit for an arbitrarily small $\epsilon$, beginning and ending at $p$, and visiting $J_f$ along the way, so (d) is proved.

(e) follows from the previous parts and $\Omega_f = J_f\cup\att(f)$ (Theorem \ref{t:main-theorem}(d)).
\end{proof}

The theorem shows that the \lq\lq Conley decomposition\rq\rq\ of an Oka-Stein manifold $X$, induced by a generic endomorphism $f$ of $X$, is as follows.
\begin{itemize}
\item  $C_f$ is partitioned into the following chain-recurrence classes.
\begin{itemize}
\item  The union of the Julia set and the non-pre-recurrent Fatou components.
\item  Each attracting cycle is a chain-recurrence class.
\end{itemize}
\item  $X\setminus C_f$ is the union of the basins of attraction of the attracting cycles with the cycles themselves removed.
\end{itemize}
Thinking of the Lyapunov function $L$ as denoting height, we can picture the union of the Julia set and the non-pre-recurrent Fatou components at constant height at the \lq\lq top\rq\rq\ of $X$.  The attracting cycles are at the \lq\lq bottom\rq\rq\ of $X$, each at a different height, and each point in a basin of attraction moves down with time towards the cycle that attracts it.

The following corollary of Theorem \ref{t:Conley-basics} is immediate.

\begin{corollary}   \label{c:equivalent-properties}
For a generic endomorphism of an Oka-Stein manifold $X$, the following are equivalent.
\begin{enumerate}
\item[(i)]  Every Fatou component is pre-recurrent.
\item[(ii)]  The Fatou set is the union of the basins of attraction of the attracting cycles.
\item[(iii)]  The robustly non-expelling set is the union of the basins of attraction of the attracting cycles.
\item[(iv)]  Every chain-recurrent point is non-wandering.
\item[(v)]  The periodic points are dense in the chain-recurrent set.
\item[(vi)]  The Julia set is a chain-recurrence class.
\item[(vii)]  The chain-recurrent set has no interior.
\end{enumerate}
If endomorphisms of $X$ satisfy the weak closing lemma for chain-recurrent points, then a generic endomorphism of $X$ has these properties.
\end{corollary}

The weak closing lemma for chain-recurrent points\footnote{This lemma states that if $p$ is a chain-recurrent point of $f\in\End\,X$, $V$ is a neighbourhood of $p$ in $X$, and $W$ is a neighbourhood of $f$ in $\End\,X$, then there is an endomorphism in $W$ with a periodic point in~$V$.} would imply that periodic points are dense in the chain-recurrent set of a generic endomorphism, and thus give an affirmative answer to the open question posed in Section \ref{sec:Fatou}, just as the weak closing lemma for non-wandering points implies that periodic points are dense in the non-wandering set of a generic endomorphism \cite[Corollary 1(b)]{AL2020}.  The closing lemma for chain-recurrent points of diffeomorphisms of compact smooth manifolds, or rather the stronger connecting lemma for pseudo-orbits, is a major result in modern smooth dynamics \cite[Th\'eor\`eme~1.2]{BC2004}.

We conclude this section by determining the topological structure of each of the six sets of endomorphisms defined by the properties (i)--(vii) in the corollary.  According to the corollary, the symmetric difference of any two of them is meagre (meaning that the complement is residual, that is, contains a dense $G_\delta$).  A subset $E$ of a topological space $Y$ is said to be nearly open or have the Baire property if there is an open subset $U$ of $Y$ such that the symmetric difference $E\bigtriangleup U$ is meagre.  A set is nearly open if and only if it is the union of a $G_\delta$ set and a meagre set or, equivalently, the complement of a meagre set in an $F_\sigma$ set.  The nearly open subsets of $Y$ form a $\sigma$-algebra, the smallest $\sigma$-algebra containing all open subsets and all meagre subsets of $Y$.  (See \cite[Section~8.F]{Kechris1995}.)

\begin{theorem}   \label{t:nearly-open}
Let $X$ be an Oka-Stein manifold.  The subsets of $\End\, X$ defined by the properties {\rm (i)--(vii)} above and their complements are nearly open.
\end{theorem}

\begin{proof}
Consider the set $U$ of pairs $(p,f)$ in $X\times \End\, X$ such that $p\in\rne(f)$.  Let $V$ be the subset of pairs $(p,f)$ such that $p$ lies in the basin of attraction of some attracting cycle of $f$.  Clearly, $U$ is open.  We claim that $V$ is also open.  Let $(p,f)\in V$ and take a small neighbourhood $W$ of the attracting cycle of $f$ to which $p$ is attracted.  Then $f^m(p)\in W$ for some $m\geq 0$.  By persistence of attracting cycles, every endomorphism in a sufficiently small neighbourhood $Z$ of $f$ has an attracting cycle in $W$ with $W$ in its basin of attraction.  The pairs $(q,g) \in X\times Z$ such that $g^m(q)\in W$ lie in $V$ and form a neighbourhood of $(p,f)$.

We conclude that $U\setminus V$ is $G_\delta$ in $\End\, X$.  Let $\pi:X\times \End\, X\to \End\,X$ be the projection.  Then $\pi(U\setminus V)$ is an analytic subset of the Polish space $\End\, X$ \cite[Section 14.A]{Kechris1995}.  It is the set of endomorphisms $f$ for which the basins of attraction do not cover all of $\rne(f)$, that is, the complement of the set defined by property (iii).  By a theorem of Lusin and Sierpi\'nski \cite[Theorem 21.6]{Kechris1995}, an analytic set in a Polish space is nearly open.
\end{proof}

By Theorem \ref{t:non-recurrent-dense}, the set of endomorphisms of $X$ with a non-pre-recurrent Fatou component is dense in $\End\,X$.  It is the complement of the subset of $\End\,X$ defined by property (i).  By Theorem \ref{t:nearly-open}, it is the union of a $G_\delta$ set and a meagre set.  If the meagre set could be taken to be empty, then we could conclude that the negations of properties (i)--(vii) are generic and consequently that the weak closing lemma for chain-recurrent points fails in our setting.


\begin{thebibliography}{88}

\bibitem{Abate2015}
M. Abate.  \textit{Fatou flowers and parabolic curves.}  Complex analysis and geometry, 1--39, Springer Proc. Math. Stat., 144.  Springer, 2015. 

\bibitem{AR2014}
M. Arizzi, J. Raissy.  \textit{On \'Ecalle-Hakim's theorems in holomorphic dynamics.}  Frontiers in complex dynamics, 387--449, 
Princeton Math. Ser., 51.  Princeton Univ. Press, 2014.

\bibitem{AL2019}
L. Arosio, F. L\'arusson.  \textit{Chaotic holomorphic automorphisms of Stein manifolds with the volume density property.}  J. Geom. Anal. \textbf{29} (2019) 1744--1762.

\bibitem{AL2020}
L. Arosio, F. L\'arusson.  \textit{Generic aspects of holomophic dynamics on highly flexible complex manifolds.}  Ann. Mat. Pura Appl. \textbf{199} (2020) 1697--1711.

\bibitem{Baker1975}
I. N. Baker.  \textit{The domains of normality of an entire function.}
 Ann. Acad. Sci. Fenn. Ser. A I Math. \textbf{1} (1975) 277--283.

\bibitem{BJ2000}
E. Bedford, M. Jonsson.  \textit{Dynamics of regular polynomial endomorphisms of $\mathbf C^k$.}  Amer. J. Math. \textbf{122} (2000) 153--212.

\bibitem{BC2004}
C. Bonatti, S. Crovisier.  \textit{R\'ecurrence et g\'en\'ericit\'e.}  Invent. math. \textbf{158} (2004) 33--104.

\bibitem{Conley1978}
C. Conley.  \textit{Isolated invariant sets and the Morse index.}  CBMS Regional Conference Series in Mathematics, 38.  Amer. Math. Soc., 1978.

\bibitem{FS1997}
J. E. Forn\ae ss, N. Sibony.  \textit{The closing lemma for holomorphic maps.}  Ergodic Theory Dynam. Systems \textbf{17} (1997) 821--837.

\bibitem{FS1998}
J. E. Forn\ae ss, N. Sibony.  \textit{Fatou and Julia sets for entire mappings in $\C^k$.}  Math. Ann. \textbf{311} (1998) 27--40.

\bibitem{Forstneric2017}
F. Forstneri\v c.  \textit{Stein manifolds and holomorphic mappings.  The homotopy principle in complex analysis.}  Second edition.  Ergebnisse der Mathematik und ihrer Grenzgebiete, 3.\ Folge, 56.  Springer, 2017.

\bibitem{FL2011} 
F. Forstneri\v c, F. L\'arusson.  \textit{Survey of Oka theory.}  New York J. Math. \textbf{17a} (2011) 1--28. 

\bibitem{Hakim1997}
M. Hakim.  \textit{Transformations tangent to the identity.  Stable pieces of manifolds.}  Preprint, 1997.

\bibitem{HM1997}
H. Hamm, N. Mihalache.  \textit{Deformation retracts of Stein spaces.}  Math. Ann. \textbf{308} (1997) 333--345.

\bibitem{Hirsch1976}
M. W. Hirsch.  \textit{Differential topology.}  Graduate Texts in Mathematics, 33.  Springer, 1976.

\bibitem{Hurley1992}
M. Hurley.  \textit{Noncompact chain recurrence and attraction.}  Proc. Amer. Math. Soc. \textbf{115} (1992) 1139--1148. 

\bibitem{Kechris1995}
A. S. Kechris.  \textit{Classical descriptive set theory.}  Graduate Texts in Mathematics, 156.  Springer, 1995.

\bibitem{RS2017}
L. Rempe-Gillen, D. Sixsmith.  \textit{Hyperbolic entire functions and the Eremenko-Lyubich class: Class $\mathcal B$ or not class $\mathcal B$?}  Math. Zeit. \textbf{286} (2017) 783--800.

\bibitem{Schleicher2010}
D. Schleicher.  \textit{Dynamics of entire functions.}  Holomorphic dynamical systems, 295--339.  Lecture Notes in Math., 1998.  Springer, 2010.

\bibitem{Touhey1997}
P. Touhey.  \textit{Yet another definition of chaos.}  Amer. Math. Monthly \textbf{104} (1997) 411--414.

\end{thebibliography}
\end{document}